\documentclass[12pt]{article}
\usepackage[top=1.3in, bottom=1.3in, left=1.1in, right=1in]{geometry}
\usepackage{latexsym}
\usepackage{amsmath}
\usepackage{amssymb,latexsym}
\usepackage{amsthm}
\newtheoremstyle{plain}
  {}
  {}
  {\itshape}
  {}
  {\bfseries}
  {.}
  { }
  {\thmname{#1}\thmnumber{ #2}\thmnote{ (#3)}}

\theoremstyle{plain}
\newtheorem{theorem}{Theorem}
\newtheorem{lemma}{Lemma}
\newtheorem{proposition}{Proposition}
\newtheorem{corollary}{Corollary}

\theoremstyle{definition}
\newtheorem{definition}{Definition}

\makeatletter
\renewenvironment{proof}[1][\proofname]{%
  \par\pushQED{\qed}\normalfont%
  \topsep6\p@\@plus6\p@\relax
  \trivlist\item[\hskip\labelsep\bfseries#1\@addpunct{.}]%
  \ignorespaces
}{%
  \popQED\endtrivlist\@endpefalse
}
\makeatother

\usepackage[all]{xy}
\usepackage[backend=bibtex,natbib=true,style=numeric,sorting=nyt,sortcites=true,hyperref=auto,backref=true,url=false,isbn=false,doi=true,firstinits=true]{biblatex}
\bibliography{iel.bib}
\usepackage{enumerate}
\usepackage[hidelinks,dvipdfmx,bookmarks=true,pdfencoding=auto]{hyperref}
\usepackage{bookmark}
\usepackage{float}
\usepackage[title]{appendix}
\usepackage{microtype}
\usepackage[USenglish]{babel}
\usepackage[singlelinecheck=false, justification=centering]{caption}
\usepackage{color}
\definecolor{darkblue}{rgb}{0.0,0.0,0.3}
 \hypersetup{colorlinks,breaklinks,
          linkcolor=darkblue,urlcolor=blue,
          anchorcolor=darkblue,citecolor=darkblue}
\usepackage{enumitem}

\newcommand{\forces}{\Vdash }
\newcommand{\nforces}{\nVdash}
\newcommand{\imp}{\rightarrow}

\newcommand{\Imp}{\Rightarrow}
\newcommand{\cImp}{\Leftarrow}
\newcommand{\BImp}{\Leftrightarrow}
\newcommand{\bimp}{\leftrightarrow}
\newcommand{\proves}{\vdash}
\newcommand{\nproves}{\nvdash}
\newcommand{\dm}{\Diamond}
\newcommand{\bx}{\Box}

\newcommand{\Intk}{\ensuremath{\mathsf{Int}_\bk}}
\newcommand{\ipc}{\ensuremath{\mathsf{IPC}}}
\newcommand{\siv}{\ensuremath{\mathsf{S4}}}

\newcommand{\sivv}{\ensuremath{\mathsf{S4V}}}
\newcommand{\bk}{\ensuremath{\mathbf{K}}}
\newcommand{\bv}{\ensuremath{\mathbf{V}}}
\newcommand{\kb}{\Diamond\bk}

\newcommand{\sft}[1]{\textsf{#1}}
\renewcommand{\gg}{\Gamma}
\newcommand{\gd}{\Delta}

\newcommand{\ig}{\in \Gamma}
\newcommand{\iel}{\ensuremath{\mathsf{IEL}}}

\newcommand{\ielm}{\ensuremath{\mathsf{IEL^-}}}

\newcounter{model}

\begin{document}

\title{Intuitionistic Epistemic Logic}

\author{Sergei Artemov$\ \ \ \ $\&$\ \ $ Tudor Protopopescu\\
{\small The CUNY Graduate Center} \\
{\small 365 Fifth Avenue, rm. 4329}\\
{\small New York City, NY 10016, USA }
}
\date{\today}
\maketitle

\begin{abstract}
We outline an intuitionistic view of knowledge which maintains the original Brou\-wer-Heyting-Kolmogorov semantics for intuitionism and is consistent with the well-known approach that intuitionistic knowledge be regarded as the result of verification. We argue that on this view co-reflection $A \imp \bk A$ is valid and the factivity of knowledge holds in the form $\bk A \imp \neg\neg A$ `known propositions cannot be false'.

We show that the traditional form of factivity $\bk A \imp A$ is a distinctly classical principle which, like {\it tertium non datur} $A\vee\neg A$, does not hold intuitionistically, but, along with the whole of classical epistemic logic, is intuitionistically valid in its double negation form $\neg\neg(\bk A\imp A)$.  

Within the intuitionistic epistemic framework the knowability paradox is resolved in a constructive manner. We argue that this paradox is the result of an unwarranted classical reading of constructive principles and as such does not have the consequences for constructive foundations traditionally attributed it. 
\end{abstract}

\section{Introduction}

Our goal is to lay the formal foundation for the study of knowledge from an intuitionistic point of view.
The resulting notions of knowledge and belief, hence, should be faithful to the intended semantics of intuitionistic logic: the  Brouwer-Heyting-Kolmogorov (BHK) semantics.
This well-established view regards \textbf{belief and knowledge as the product of verification}. 

While the standard domain of our theory is the same as that of BHK -- mathematical statements, proofs and verifications -- we aim to show that BHK and the resulting intutionistic systems of epistemic logic, \ielm\ and \iel, yield principles of constructive epistemic reasoning which apply in more general settings. 
This framework also offers a natural resolution of the Church-Fitch knowability paradox, and suggests a more accommodating formal basis for an intuitionistically inspired philosophical verificationism than plain intuitionistic logic.

Intuitionistic belief and knowledge behave quite differently from their classical counterparts, and while comparisons are helpful and apposite it must be kept in mind that assumptions and distinctions which make sense classically may not intuitionistically, or take a different form.

The fundamental difference between intuitionistic and classical knowledge lies in their relationship to their respective notions of truth. 

According to the BHK semantics an intuitionistic proposition is true if proved. 
Since intuitionistic belief and knowledge is the product of verification, the intuitionistic truth of a proposition is sufficient for both belief and knowledge because intuitionistic truth contains proof and every proof is also a verification: 

  \begin{center}
      \emph{Intuitionistic Truth $\ \ \Imp\ $ Intuitionistic Knowledge}.
  \end{center}
\noindent This insight is fundamental to the nature of intuitionistic reasoning about epistemic propositional attitudes, and how it differs from classical epistemic reasoning. Intuitionistically the principle of the \emph{constructivity of truth}, a.k.a \emph{co-reflection} ($\bk$ is the knowledge modality)
\[
A \imp \bk A\label{CT}\tag{co-reflection}
\]
is a truism about both belief and knowledge -- for the aforementioned reason that all proofs are verifications. 
Classically, of course, it is invalid because it asserts a form of omniscience, that all classical truths are classically known.

What about the truth condition on knowledge, `known propositions are true', in the intuitionistic setting?  Classically, 
this yields the \emph{factivity of knowledge} which has the logical form of the \emph{reflection principle}  $\bk A \imp A$. 
However, in the intuitionistic setting, reflection is too strong, to the extent of being invalid. 
The verification-based approach allows that justifications more general than proof can be adequate for belief and knowledge, e.g.\ verification by trusted means which do not necessarily produce explicit proofs of what is verified. According to this view, the reflection principle for intuitionistic knowledge is not universally valid: it is possible to have a provably verified proposition $A$ without possessing a specific proof of $A$ itself (cf.\ section~\ref{sec:knp}). 

On the other hand, the truth condition for knowledge in the form  `known propositions cannot be false' is intuitionistically valid and produces the principle 
\[ 
\phantom{intuitionistic reflection}\ \ 
\bk A \imp \neg\neg A .\tag{intuitionistic reflection}
\]
Indeed, if $\bk A$ then it is verified that $A$ has a proof, not necessarily specified in the process of verification; from this we conclude that it is not possible to produce a proof that $A$ cannot have a proof, hence $\neg\neg A$.
Naturally, in the classical framework, reflection and intuitionistic reflection are equivalent, and adopting the former instead of the latter is harmless, but not in the intuitionistic setting.
\footnote{The double negation translation (\cite{Dosen1984,Kolmogorov1925,Glivenko1929,Heyting1966,vanDalen1988v1,vanDalen2002,Chagrov1997}) which, for atomic $A$ is $\neg\neg A$, is a canonical way to approximate the classical truth of $A$ intuitionistically. 
Whereas the BHK requirement for the intuitionistic truth of $A$ is to have a proof of $A$, the formula $\neg\neg A$ can be intuitionistically true without an explicit proof of $A$; truth in this sense may be regarded as some form of classical truth.}

One should not expect all classical logical laws to stay valid intuitionistically. Many classical principles cannot be transplanted into the intuitionistic domain as they are, e.g.\ $A \lor \neg A$, $\neg\neg A \imp A$, $((A \imp B) \imp A) \imp A$, etc.\
These are all classical tautologies, not valid intuitionistically 
without natural adjustments that provide them with appropriate constructive meaning. 
There are many ways to make these formulas intuitionistically acceptable without changing their classical reading, e.g., by Glivenko's Theorem, \cite{Glivenko1929}: 
\[
\mathsf{CPC} \proves A \ \BImp \ \ipc \proves \neg\neg A 
\]
(here {\sf CPC} denotes {\it classical propositional logic} and {\sf IPC} -- {\it intuitionistic propositional logic}, cf.\ \cite{vanDalen1988v1}). 
This means that for each classical tautology $A$, the formula $\neg \neg A$ is a valid intuitionistic principle.
\footnote{\label{ft:Kol}Another canonical way of embedding classical logic principles into the intuitionistic domain is Kolmogorov's \cite{Kolmogorov1925} double negation translation ``$\neg\neg$ each subformula". }
It turns out that the reflection of classical knowledge $\bk A \imp A$ is from the same cohort: it is valid classically, while not valid intuitionistically, but its double negation translation $\neg \neg (\bk A \imp A)$ is intuitionistically valid and may be adopted as an intuitionistic form of reflection for knowledge. 
In \iel\ we have opted for its equivalent version $\bk A \imp \neg \neg A$, which has an even more vivid factivity reading.
In this respect intuitionistic reflection can be read as claiming that intuitionistic knowledge yields truth \emph{but without an explicit proof of that truth}. 
Therefore, what classical factivity expresses is preserved by intuitionistic reflection. 

So, intuitionistic knowledge of $A$ is positioned strictly in between $A$ and $\neg\neg A$:
\[ A\imp \bk A\imp \neg\neg A \]
which provides a basis for a more refined analysis of constructive truth than plain intuitionistic logic. 
If we assume that the double negation translation is a meaningful intuitionistic representation of classical truth, these findings can be presented as
 \begin{center}
      \emph{Intuitionistic Truth $\ \ \Imp\ $ Intuitionistic Knowledge $\ \ \Imp\ $ Classical Truth}.
  \end{center}

\subsection{Logics of Intuitionistic Belief and Knowledge}

Extending the BHK semantics with the notion of verification, and conceiving of intuitionistic belief and knowledge as a result of it, yields intuitionistic systems of epistemic logic, \ielm, \iel.
The key property of these systems, and hence of intuitionistic belief and knowledge, in contrast to the classical, is that they all validate co-reflection, and can distinguish between different strengths of reflection, or the truth condition.
\footnote{\label{fn:FK} Though all are compatible with reflection, endorsing reflection in an intuitionistic setting would represent a restrictive proof-based view of knowledge and trivialize the resulting epistemic logic system. See the end of section \ref{sec:bhkk}.} 

We begin with a general discussion of intuitionistic, verification-based, belief and knowledge, and the principles which distinguish them from each other and from their classical counterparts (section \ref{sec:bhkk}).
The intuitionistic validity of co-reflection for belief and knowledge changes the situation dramatically, and intuitionistic knowledge is not distinguishable from belief in the same way as classical knowledge from classical belief.
\medskip\par
{\bf The basic intuitionistic logic of belief} \ielm\ is given by the epistemic closure principle $$\bk(A\imp B)\imp(\bk A\imp\bk B)$$ along with the adoption of co-reflection \[ A\imp\bk A,\] which states that intuitionistic beliefs respect BHK-proofs: if $A$ is constructively true, i.e.\ has a specific proof, then the agent knows/believes that $A$. 
In \ielm, theoretically, false beliefs are not {\it a priori} ruled out.

As mentioned above, the intuitionistic truth condition on knowledge, i.e.\ intuitionistic factivity, `known propositions cannot be false' admits a formalization as intuitionistic reflection  $$\bk A\imp\neg\neg A.$$ 
Adding intuitionistic reflection to \ielm\ leads to the system \iel, which is {\bf the logic of intuitionistic knowledge}. 
We will see that in the intuitionistic context, in the presence of co-reflection, intuitionistic reflection as well as other natural alternatives of the intuitionistic truth condition are equivalent to provable consistency $\neg \bk \bot$ (section \ref{sec:knf}, appendix \ref{sec:itcond}) therefore, \iel, is both {\bf the logic of intuitionistic knowledge} and {\bf the logic of provably consistent intuitionistic beliefs}. 
 
We prove soundness and completeness of \ielm\ and \iel\ with respect to appropriate classes of self-explanatory Kripke-style models and derive some notable epistemic principles (section \ref{sec:mod}, appendix \ref{sec:comp}). 

We also address the knowability paradox and intuitionistic responses to it (section \ref{sec:kbpdx}).
Our framework yields a well-founded constructive resolution which shows that the knowability paradox is the product of an unwarranted classical reading of constructive principles and does not have the consequences for constructive foundations traditionally attributed it.

Finally, we address intuitionistic counter-arguments to co-reflection, and to the very idea of intuitionistic knowledge arising from criticism of intuitionistic responses to the knowability paradox.
We argue that none of these arguments are well-founded because their view of intuitionistic knowledge is not intuitionistic enough (section \ref{sec:crit}). 

\section{The Brouwer-Heyting-Kolmogorov Semantics and Knowledge}\label{sec:bhkk}

The Brouwer-Heyting-Kolmogorov semantics for intuitionistic logic (cf.\ \cite{vanDalen1988v1}) holds that a proposition, $A$, is true if there is a proof of it, and false if we can show that the assumption that there is a proof of $A$ leads to a contradiction. Truth for the logical connectives is defined by the following clauses: 
\begin{itemize}
\item a proof of $A \land B$ consists in a proof of $A$ and a proof of $B$;
\item a proof of $A \lor B$ consists in giving either a proof of $A$ or a proof of $B$;
\item a proof of $A \imp B$ consists in a construction which given a proof of $A$ returns a proof of $B$;

\item $\neg A$ is an abbreviation for $A \imp \bot$, and $\bot$ is a proposition that has no proof. 

\end{itemize}

Our question is: if we add an epistemic operator $\bk$ to our language, what should be the intended semantics of a proposition of the form $\bk A$? 

We adopt the view that \textbf{an intuitionistic epistemic state (belief or knowledge) is the result of verification}, where a verification is evidence considered sufficiently conclusive for practical purposes.
\footnote{\label{fn:vercite} Verifications, hence, are not necessarily generalizations of the notion of `canonical proof' found in philosophical verificationism, see e.g.\ \cite{Contu2006, Dummett1973, Dummett1963, Dummett1976, Dummett1977, Dummett1991, Martin-Lof1998, Prawitz1980, Prawitz1998d, Prawitz2005, Prawitz2006, Schroeder-Heister2006, Sundholm2002, Tennant1997, Usberti2006, vanDalen1988v2}. 
A verification does not have to be canonical or even a means for acquiring a canonical verification, consider the examples in section \ref{sec:knp}.
For a similar reading of $\bk$ in a more formal setting see Williamson's proposal for an intuitionistic epistemic logic \cite{Williamson1992}, see section \ref{sec:williamsoniel}.
See \cite{Wright1982} for some discussion of the nature of verification and its relation to a generalised intuitionism.} 
The idea that verifications are not necessarily proofs is common in the verificationist literature, see the citations in note \ref{fn:vercite}.
In a formal setting Williamson incorporates non-proof verifications in his system, \cite{Williamson1992}; see also note \ref{ft:Wmson}. 

\subsection{The BHK Clause for the Knowledge Modality}

Before we introduce a $\bk$-clause, we assume that the conception of proof has two salient features: (1) proofs are conclusive of the proposition they establish and hence are (the purest form of) verifications; (2) proofs are checkable -- that something is a proof is itself capable of proof.

We propose the following epistemic BHK clause governing the knowledge operator $\bk$:
\begin{itemize}
\item a proof of $\bk A$ is conclusive evidence of verification that $A$ has a proof.
\end{itemize}
Such a verification, of course, need not deliver a proof of $A$ itself.
For example, consider propositions from the point of view of Intuitionistic Type Theory, \textsf{ITT}.
Propositions are special types whose elements are proofs (or evidence or witnesses), that is each inhabitant of the proposition type may be considered to be a proof of the proposition.
For each type $A$ one can form a `truncated type' $inh(A)$,
\footnote{or $||A||$.} 
see \cite{HoTT2013} (also called `squash types', `monotypes' or `bracket types', see \cite{Constable1998,Kopylov2001}), which contains no information beyond the fact that the type $A$ is inhabited.
Where the inhabitants of the proposition type $A$ are proofs of $A$, which each convey some specific information, $inh(A)$ ``forgets'' all such information other than the existence of these inhabitants. 
A truncated type can have at most one inhabitant which serves only to indicate either `yes' or `no' depending on whether $A$ is indeed inhabited.
A truncated proposition of type $inh(A)$, hence, conveys only the information that $A$ has a proof, but does not deliver any such proof of $A$.
We can interpret $\bk A$ as just such a truncated type $inh(A)$. 

Even in this verbal form, the $\bk$-clause suffices for some epistemic analysis. For example, we can check that $\bk (A \vee B)$ does not necessarily yield that $\bk A$ holds or that $\bk B$ holds 
\[ \bk (A\vee B)\ \not\imp \ \bk A \vee \bk B ;\] 
Indeed, $\bk (A\vee B)$ states that there is a proof of disjunction $A\vee B$, but does not actually produce such a proof, or even a means for constructing such a proof, so we cannot decide which of  $\bk A$ and $\bk B$ holds, this is confirmed by Theorem \ref{th:nkdisj} below. 

\subsection{The BHK Meaning of Knowledge Assertions}
\label{sec:bhk-meaning}

By the BHK clause above $\bk A$ is read as 
\emph{it is verified that $A$ holds intuitionistically}, i.e.\ that $A$ has a proof, not necessarily specified in the process of verification. However, the intuitionistic epistemic logic we construct, section \ref{sec:IEL}, also captures a reading:
\begin{equation}
\mbox{ \emph{it is verified that $A$ holds in some not specified constructive sense.}}\label{second}\tag{*}
\footnote{I.e.\ is not necessarily a BHK-compliant proof, but constructive in a more general sense, see the examples of highly probable truth, and empirical knowledge in section \ref{sec:knp} below.
The examples of zero-knowledge protocols, testimony of an authority, existential generalization, and classified sources can be read as examples of the first reading.} 
\end{equation}
The language of intuitionistic logic is not sensitive enough to distinguish these readings.
We regard the first reading as our `official' one, and leave exploration of the second to later investigations.
\footnote{For instance a bi-modal classical language can distinguish and capture these readings, see section \ref{sec:embed}.} 

A question to consider is whether a proposition is intuitionistically true only if an agent is aware of a proof or whether the possibility of such awareness is enough? 
Traditionally, intuitionism assumes that proofs are available to the agent. 
For Brouwer and Heyting proofs are mental constructions,
\footnote{Brouwer \cite[p.4]{Brouwer1981} considered intuitionistic mathematics to be ``an essentially languageless activity of the mind''. Heyting \cite[p.2]{Heyting1966} says ``In the study of mental mathematical constructions `to exist' must be synonymous with `to be constructed'". See also \cite{Heyting1930, Heyting1964}.} 
and so the existence of a proof requires its actual construction. 
This position is the one traditionally adopted by verificationists, see e.g.\ \cite{Dummett1973, Dummett1977, Dummett1998}. 
On the other hand Prawitz \cite{Prawitz1980, Prawitz1998, Prawitz1998d, Prawitz1998e, Prawitz1998f, Prawitz1998g} and Martin-L\"of \cite{Martin-Lof1987, Martin-Lof1998, Martin-Lof1990}, consider proofs to be timeless entities, and that intuitionistic truth consists in the existence of such proofs, and their potential to be constructed. 

The principles of intuitionistic knowledge and belief we discuss below are compatible with either of these positions. 
Hence, if BHK proofs are assumed to be available to the agent, then $\bk A$ can be read as `$A$ is believed' or `$A$ is known', depending on the assumptions made about the epistemic state.  
If proofs are platonic entities, not necessarily available to the knower, then $\bk A$ is read as `$A$ can be believed (or known) under appropriate conditions'. 
To keep things simple, in our exposition we follow the former, more traditional, understanding.
So to claim $\bk A$ is true is to claim that the agent is aware of a proof that it has been verified that $A$ has a proof. 

Co-reflection does not hold for the combination of the above two positions where proofs are timeless platonic entities while knowledge requires actual awareness; the existence of a proof does not guarantee an agent is aware of it. 
This combination rather validates a form of knowability, if a proposition has a proof then it is possible to know it, i.e.\ \[A \imp \kb A\] (see section \ref{sec:kbpdx} for discussion of this and other possible formalizations of the knowability principle).

\subsection{Principles of Intuitionistic Epistemic Logic}\label{sec:pik}

Co-reflection and reflection play a special role in intuitionistic epistemology. Co-reflection is the defining principle of intuitionistic epistemic states, it holds for any shade of intuitionistic belief or knowledge. 
Reflection, on the other hand, holds only for epistemic states with degenerated knowledge at which  `$A$ is known' is equivalent to `$A$' itself.  
This, of course, is very different from the classical case where reflection is the defining principle of knowledge and co-reflection holds only for omniscient epistemic states. 

These facts follow from the assumption that intuitionistic belief and  knowledge is the result of verification, and the fact that intuitionistic truth is based on proof. 
Given this, reflection and co-reflection may be seen as expressing two informal principles about the relationship between truth and verification-based knowledge:
\begin{enumerate}
  \item proof yields verification-based knowledge (co-reflection);
  \item verification-based knowledge yields proof (reflection).
  \end{enumerate}
A BHK-compliant epistemology accepts 1 and rejects 2. 

\subsubsection{Proof Yields Belief and Knowledge}\label{sec:pyk}

The principle that proof yields verification is practically constitutive of proof,
\footnote{Though not common in mainstream epistemology there are, or have been, mathematical skeptics. Perhaps the best known mathematical skeptical argument is the one Descartes puts forward in \cite{Descartes1642}, see also \cite{Floridi1998, Floridi2000, Floridi2004}. See also \cite{Frege1884} and \cite{Husserl1901} who both discuss the skeptical consequences of empiricism regarding mathematical knowledge.} 
precisely because \textbf{proofs are a special and most strict kind of verification}, and immediately justifies the validity of the formal principle of co-reflection\[ A\imp \bk A.\] 

That proofs are taken to be verifications is a matter of the ordinary usage of the term which understands a proof as ``an argument that establishes the validity of a proposition'' \cite{Proof}. 
It is also a fairly universal view in mathematics (cf. \cite{Buss1998,Putnam1977,Tarski1969}). 
Within computer science this concept is the cornerstone of a big and vibrant area in which one of the key purposes of computer-aided proofs is for the verification of the propositions in question \cite{Clarke1996,Constable1998}. 
Amongst intuitionists the idea of a constructive proof is often treated as simply synonymous with verification \cite{Dummett1970, Dummett1979, Horsten2014}.
Hence co-reflection should be read as expressing the constructive nature of intuitionistic truth, which itself being a strict verification yields verification-based belief or knowledge.

According to the BHK reading of intuitionistic implication co-reflection states that \textbf{given a proof of $A$ one can always construct a proof of $\bk A$.} 
Is such a construction always possible? 
Indeed, it is well established that proof-checking is a valid operation on proofs,
\footnote{See \cite{Hilbert1932, Kreisel1962, Lob1955, Godel1933, Artemov2001}. Moreover, proof-checking is generally a feasible operation, routinely implemented in a standard computer-aided proof package. 
} 
so if $x$ is a proof of $A$ then it can be proof-checked and hence produce a proof $p(x)$ of `$x$ is a proof of $A$'. 
Having checked a proof we have a proof that the proposition is proved, hence verified, hence known or believed.
In whatever sense we consider a proof to be possible, or to exist, co-reflection states that the proof-checking of this proof is always possible, or exists, in the same sense. 
So, by the principle that proof yields verification we have that a proof of $A$ yields knowledge or belief of $A$, and by proof-checking we obtain a proof of $\bk A$.

On the type-theoretic reading of $\bk A$ co-reflection is immediate; given a specific inhabitant of the proposition $A$ it is guaranteed that the type is inhabited, hence $inh(A)$ holds.  

We are not, of course, the first to outline arguments that an intuitionistic conception of truth supports co-reflection, see for instance \cite{Dean2009, Hart1979, Khlentzos2004, Murzi2010, Percival1990, Usberti1994, Williamson1982, Williamson1988, Wright1993a}.
\footnote{Cf. \cite[p.90]{Usberti1994}: ``\ldots [co-reflection] can be interpreted only according to the intuitionistic meaning of implication, so that it expresses the trivial observation that, as soon as a proof of $A$ is given, $A$ becomes known''.} 
Our contention is that co-reflection, when properly understood in line with the intended BHK semantics, is a fairly immediate consequence of uncontroversially intuitionistic views about truth, and {\bf should therefore be endorsed as foundational for a properly intuitionistic epistemology}.

\subsubsection{Knowledge Does Not Yield Proof}\label{sec:knp}

If not all verifications are BHK-comp\-liant proofs then it follows that verification-based knowledge does not necessarily yield proof, and consequently that reflection is not a valid intuitionistic epistemic principle. 
It is possible, hence, to have knowledge of a proposition without possessing its proof, i.e.\ without it being intuitionistically true.
\footnote{\label{ft:Wmson} Cf. \cite[p.68]{Williamson1992} ``\dots $\bk$ requires more than warranted assertion. However, it does not follow that $\bk$ requires strict proof; that would not be a reasonable requirement when $\bk$ is applied to  empirical statements\dots''.} 

The BHK reading of reflection $\bk A\imp A$ says that \textbf{given a proof of $\bk A$ one can always construct a proof of $A$}, that is it asserts that there is a uniform procedure, or construction, which given a proof of $\bk A$ returns a proof of $A$ itself. 
Since we allow that $\bk A$ does not necessarily produce specific proofs this requirement is not met for intuitionistic knowledge, and \emph{a fortiori} for belief.
What uniform procedure is there that can take any adequate, non-proof, verification of $A$ and return a proof of $A$? 
There is no such construction. Consider the following counter-examples against the factivity of intuitionistic verification-based knowledge. 

The following four examples correspond to the principal reading of $\bk A$ as {\emph{$A$ has a proof, not necessarily specified in the process of verification}:

\paragraph{Zero-knowledge protocols}\label{sec:zero}

A class of cryptographic protocols, normally probabilistic, by which the prover can convince the verifier that a given statement is true, without conveying any additional information apart from the fact that the statement is true. The canonical way these protocols work is that the prover possesses a proof $p$ of $A$, and convinces the verifier that $A$ holds without disclosing $p$. 

\paragraph{Testimony of an authority}\label{sec:testimony}
 
Even concerning mathematical knowledge reflection fails. Take Fermat's Last Theorem. For the educated mathematician it is credible to claim that it is known, but most mathematicians could not produce a proof of it. Indeed, more generally, any claim to mathematical knowledge based on the authority of mathematical experts is not intuitionistically factive. It is legitimate to claim to know a theorem when one understands its content, and can use it in one's reasoning, without being in a position to produce or recite the proof. 

\paragraph{Classified sources}\label{sec:classified}

In a social situation, imagine a statement of $A$ coming from a most reliable source but with a classified origin. So, there is no access to the `strict proof' of $A$. Should we abstain from reasoning about $A$ as something known unless we gain full access to the strict proof? This is not how society works. We treat $\bk A$ as weaker than $A$, and keep reasoning constructively without drawing any conclusion about a specific proof of $A$.  

\paragraph{Existential generalisation}\label{sec:existential}

Somebody stole your wallet in the subway. You have all the evidence for this: the wallet is gone, your backpack has a cut in the corresponding pocket, but you have no idea who did it. You definitely know that `there is a person who stole my wallet" (in logical form, $\exists x S(x)$, where $S(x)$ stands for `$x$ stole my wallet') so you have a justification $p$ of $\bk(\exists x S(x))$. If $\bk(\exists x S(x)) \imp \exists x S(x)$ held intuitionistically, you would have a constructive proof $q$ of $\exists x S(x)$.  However, a constructive proof of the existential sentence $\exists x S(x)$ requires a witness $a$ for $x$ and a proof $b$ that $S(a)$ holds. You are nowhere near meeting this requirement. So, $\bk(\exists x S(x)) \imp \exists x S(x)$ does not hold intuitionistically.
\bigskip\par
Here are examples which can be captured by the broader constructive reading of $\bk A$ (\ref{second}): 

\paragraph{Highly probable truth}\label{sec:probable}

Suppose there is a computerized probabilistic verification procedure, which is constructive in nature, that supports a proposition $A$ with a cosmologically small probability of error, so its result satisfies the strictest practical criteria for truth. Then any reasonable agent accepts this certification as adequate justification of $A$, hence $A$ is known. Moreover, observing the computer program to terminate with success, we have a proof that $\bk A$. However, we do not have a proof of $A$ in the sense required by BHK; we cannot even claim that such a proof exists.

\paragraph{Empirical knowledge}\label{sec:perceptual}

Suppose that some phenomenon has been repeatedly observed under optimal experimental conditions. 
After a certain number of repeated observations these are taken to confirm some hypothesis, $A$, that predicted the phenomenon. 
For practical scientific purposes these observations are a verification of the hypothesis, hence it is legitimate to claim $\bk A$, but there is no reason to claim having a BHK-compliant proof of this.

\bigskip

If we allow that knowledge may be gained by any of the methods above then reflection is not valid according to the BHK semantics.

How might we be in a position of having a proof of $\bk A$ without thereby being in a position to obtain a proof of $A$ itself? 
Consider the example of zero-knowledge protocols, which by design yield verifications of statements without disclosing any further information about them. 
Given a proof that $A$ is verified in this manner is there a general method for constructing a proof of $A$ itself from this information?
Clearly not. 
This is because, in general, claiming 
\begin{center}
 \emph{it is proved that $A$ is verified}
\end{center}
is a weaker statement than 
\begin{center}
 \emph{it is proved that $A$}.
\end{center}
All the former statement gives us is a guarantee that $A$ has a verification, which by assumption does not necessarily yield an explicit proof. 

The invalidity of the straightforward constructive reading of the classical reflection $\bk A\imp A$ is particularly evident on the type-theoretical interpretation of $\bk A$.
Given a truncated proposition $inh(A)$ all the information one has is that $A$ is inhabited, hence one knows there is a proof, but $inh(A)$   does not deliver any such inhabitant, hence one is not in a position to assert $A$ since one cannot produce an element of the type.

\subsubsection{Intuitionistic Reflection and the Truth Condition on Knowledge}\label{sec:knf}

Nevertheless reflection has often been taken to be practically definitive of knowledge from a constructive standpoint.
For instance, Williamson \cite{Williamson1992}, in outlining his system of intuitionistic epistemic logic affirms that $\bk A \imp A$  holds. 
Similarly Proietti, \cite{Proietti2012}, argues that knowledge is factive in his system of intuitionistic epistemic logic. 
Wright states that an operator could not be a knowledge operator if it were not factive \cite{Wright1994}.
\footnote{``I take this to be a non-negotiable feature of the concept of knowledge. If a theory takes a view of something which it purports to regard as knowledge, but which lacks this feature, it is not a theory of knowledge'' \cite[p.242]{Wright1994}.} 
More generally still the principle $\bk A \imp A$  is probably the only principle about knowledge that has not been seriously contested.
\footnote{Hazlett's \cite{Hazlett2010, Hazlett2012} would seem to be the only such challenge. However Hazlett challenges the idea that the truth of $A$ is necessary for the truth of the utterance `S knows $A$'; utterances of `S knows $A$' may be true even if $A$ is false. He is careful to distinguish this challenge from the claim that it is possible to know false propositions -- that he does not challenge. Hazlett's arguments do not appear to be relevant to our concerns since we are not occupied with the truth conditions of utterances of knowledge ascriptions, but the logical analysis of the epistemic operator.} 
And, of course, it is implied by virtually every extant definition of knowledge. 
Must not our arguments above be wrong in some fashion? 
Are we not arguing the intuitionist is committed to holding that false propositions can be known? 
No, on the contrary: {\bf intuitionistic reflection $$\bk A\imp\neg\neg A$$ is classically equivalent to reflection and hence is acceptable both classically and intuitionistically}. 

Every analysis of knowledge accepts the truth condition on knowledge, that only true propositions can be known and that false propositions cannot be known. 
It is this, and not reflection \emph{per se}, which is definitive of knowledge. 
An intuitionistic formalization of the truth condition for knowledge is the principle of intuitionistic reflection which can be read as
\begin{center}
\emph{if $A$ is known then it is impossible that $A$ is false.}
\end{center}
An attempt to rewrite it into the `simpler' form $\bk A\imp A$ fails intuitionistically: reflection is strictly stronger intuitionistically than intuitionistic reflection and, as we argue, is not intuitionistically valid. 

Intuitionistic reflection can be interpreted from two perspectives, both illuminating with respect to what it captures.

On the one hand, intuitionistic reflection says that to have knowledge is to have a verification which rules out the possibility of the intuitionistic falsehood of $A$, that is the possibility of a disproof of $A$.
\footnote{See \cite{Wansing2010} for an extension of BHK by dually
refuted falsehood.} 
Intuitionistic knowledge, hence, establishes the logical possibility of the intuitionistic truth of a proposition.
Where classical knowledge guarantees the (classical) truth of a proposition, intuitionistic reflection guarantees the possibility of proof (this is made vivid in the Kripke semantics developed below, see section \ref{sec:kvsbel}). 

On the other hand, via the embedding of classical logic into intuitionistic logic we see that intuitionistic reflection expresses just what classical reflection does. 
The double negation translation of classical logic into intuitionistic logic (cf. \cite{Glivenko1929,Chagrov1997,Kolmogorov1925,vanDalen2004}) suggests the informal intuitionistic reading of $\neg\neg A$ as `$A$ is classically true'.
From this point of view intuitionistic reflection expresses that intuitionistic knowledge yields classical truth, i.e.\ knowledge of $A$ yields the truth of $A$ but \emph{without a specific witness.} 
A verification yielding knowledge provides sufficient information to claim the proof-less truth of $A$, which is just what reflection claims classically. 
The double negation embedding of \sft{CPC} into \sft{IPC} extends to the classical reflection principle -- principle 4 below, which is equivalent to intuitionistic reflection -- accordingly {\bf intuitionistic reflection expresses as much as its classical counterpart does}. 
Though classical reflection does not hold intuitionistically nothing is lost, the intuitions that support classical reflection can be captured in an intuitionistic setting.

Here is a list of other intuitionistically meaningful logical ways to express the truth condition: 
    
\begin{enumerate}[noitemsep]

   \item $\neg A \imp \neg \bk A$;

   \item $\neg(\bk A \land \neg A)$;

   \item $\neg\bk\bot$;

   \item $\neg\neg(\bk A \imp A)$.
\end{enumerate}

Intuitionistically 1, 2 and 3 can be considered as directly saying that knowledge of falsehood is impossible. 
4 can be seen as saying that reflection is classically valid, or alternatively that it is logically possible for verification to yield proof.

Principles 1, 2 and 4 are classically equivalent to reflection, and all 1 -- 4 are intuitionistically strictly weaker than reflection. 
In this way we see that intuitionistically reflection is not required in order to maintain the truth condition on knowledge, or to distinguish belief from knowledge.

As we will see, in the presence of co-reflection all 1 -- 4 and intuitionistic reflection are equivalent.
Theoretically 3 is the simplest, but we pick intuitionistic reflection since it most clearly expresses the intuitionistic notion of factivity (see section \ref{sec:IEL} and appendix \ref{sec:itcond} for more on the relation between these expressions of the intuitionistic truth condition). 

In the absence of the truth condition we do not rule out that intuitionistically verified propositions may be false.
For example, before the European discovery of Australia all available evidence supported the proposition `all swans are white'; this turned out to be false and can be taken as an instance of verification-based belief which may be false, which is captured by a logic without intuitionistic reflection.

Given the equivalences noted above, the truth condition asserts a kind of provable consistency.
For example, while the proposition that `all swans are white' does not imply a contradiction, and hence is consistent, this does not imply that it is possible to know false consistent propositions.
The truth condition requires that the consistency of a proposition be provable.
Since `all swans are white' is not provably consistent it is ruled out as knowledge by the truth condition.

\section{Intuitionistic Epistemic Logic}\label{sec:IEL}

We are now in a position to define the systems of intuitionistic belief and knowledge, \ielm\ and \iel\ respectively.
These systems, we argue, respect the intended BHK meaning of intuitionism and incorporate a reasonable verification-based epistemic operator.

The language is that of intuitionistic propositional logic augmented with the epistemic propositional operator $\bk$.
The simplest system, \ielm, {\it the logic of intuitionistic beliefs}, is given by:

\begin{definition}[\ielm] $ $ \
\medskip\par
\noindent\textbf{Axioms}

1. Axioms of propositional intuitionistic logic;

2. $\bk(A \imp B) \imp (\bk A \imp \bk B)$; \hfill (distribution)

3. $A \imp \bk A$. \hfill (co-reflection)

\noindent\textbf{Rule} {\em Modus Ponens}

\end{definition}

\medskip
The next system we consider is {\em the logic of intuitionistic knowledge},  \iel\ (which is, at the same time, the {\em logic of provably consistent intuitionistic beliefs}): 
\medskip\par
\begin{definition}[\iel]
$ \iel\ = \ielm\ + 
\bk A \imp \neg \neg A$
\hfill {\em (intuitionistic reflection)}
\end{definition}

Immediately from these definitions, we conclude 
\[ \ielm\ \subseteq\ \iel. 
\]
From model-theoretical considerations in section~\ref{sec:mod} it follows that this inclusion is strict (Theorem \ref{th:strictincl}). 

\begin{proposition}\label{thm:basicprop} 
In $\mathcal{L} \in \{\ielm, \iel\}$
  \begin{enumerate}[noitemsep]
   \item The rule of \bk-\emph{necessitation}, $\proves A\ \Rightarrow\ \proves \bk A$, is derivable. 

  \item The deduction theorem holds.

  \item Uniform substitution holds.

  \item $\mathcal{L}$ is a normal intuitionistic modal logic.
\footnote{See \cite{Bozic1984, Bozic1985}.} 

  \item Positive and negative introspection hold; $\ \proves \bk P \imp \bk\bk P$, $\ \proves \neg\bk P \imp \bk\neg\bk P$.
  
  \end{enumerate}
\end{proposition}

\begin{proof} \

\begin{enumerate}[noitemsep]
  \item By co-reflection.

  \item From 1, and the fact that intuitionistic propositional logic validates the deduction theorem.

  \item By induction on the complexity of formulas.

  \item From 3 $\mathcal{L}$ is closed under substitution for propositional variables.

  \item Both are instances of axiom $A \imp \bk A$, with $\bk P$ and $\neg\bk P$ for $A$ respectively.

\end{enumerate}  
\end{proof}

\begin{proposition}\label{th:distconj}
  For $\mathcal{L} \in \{ \ielm, \iel \}$ \[ \mathcal{L} \proves \bk(A \land B) \bimp (\bk A \land \bk B). \]
\end{proposition}

\begin{proof} \ The standard derivation of this fact uses distribution and necessitation, both present in ${\mathcal L}$. 

\end{proof}

\begin{theorem}[Truth Condition]\label{th:ielth}
\iel\ proves

\end{theorem}

\begin{enumerate}[noitemsep]

   \item $\neg \bk \bot$;

   \item $ \neg(\bk A \land \neg A)$;

   \item $ \neg A \imp \neg \bk A$;

   \item $ \neg\neg(\bk A \imp A)$.

  \end{enumerate}

\begin{proof} $\ $
\par
\noindent
For 1:
\medskip\par
1. $\bk \bot \imp \neg \neg \bot$ - intuitionistic reflection;

2. $\neg\neg\bot \imp \bot$ - {\sf IPC} theorem;

3. $\neg \bk \bot$ - from 1 and 2.
\medskip\par\noindent
For 2:
\medskip\par
1. $\bk A \land \neg A$  - assumption;

2. $\neg\neg A \land \neg A$ - by co-reflection;

3. $\bot$ - from 2;

4. $\neg(\bk A \land \neg A)$. 

\medskip\par\noindent
For 3:
\medskip

1. $\neg\neg\neg A \imp \neg \bk A$ - contrapositive of intuitionistic reflection;

2. $\neg A \imp \neg \bk A$ - by $\neg X \bimp \neg\neg\neg X$.
\medskip\par\noindent

For 4, continue with: 
\medskip\par

3. $\neg\neg \bk A \imp \neg \neg A$ - contrapositive of 2;

4. $\neg\neg(\bk A \imp A)$ - by intuitionistic tautology $(\neg\neg X \imp \neg\neg Y) \bimp \neg\neg(X \imp Y)$.
\end{proof}

It is easy to check that \ielm\ with each of $\neg\bk\bot$, $\neg(\bk A \land \neg A)$, $\neg A \imp \neg \bk A$, $\neg\neg(\bk A \imp A)$,  as additional axioms is equivalent to \iel. 
Since each of these principles can be regarded as expressing the truth condition on knowledge, we see that the axiom $\bk A \imp \neg \neg A$ is an adequate intuitionistic expression of this idea.
\footnote{See also appendix \ref{sec:itcond}.}

Note that as a corollary of part 4 of the theorem above, and Glivenko's Theorem, the classical logic of knowledge {\sf S5} as well as logics of belief    \sft{K, D, KD4, KD45} can be Glivenko-embedded into \iel: the double negation of each theorem of these logics is derivable in \iel.
This embedding, however, is not faithful; obviously $\iel \proves \neg\neg(A \imp \bk A)$ but in none of the classical logics just mentioned is it the case that $\proves A \imp \bk A$.
\footnote{The same holds for the Kolmogorov embedding, see note \ref{ft:Kol}.} 
This makes more precise the claim above that \iel\ offers a more general framework than the classical epistemic one; classical epistemic reasoning is sound in \iel, but the intuitionistic epistemic language is rather more expressive.
\footnote{Proof theory for \iel\ has been developed in \cite{Krupski2016}, which established cut-elimination theorems and demonstrated that it is PSPACE complete.}  

\subsection{Intuitionistic Knowledge as Provably Consistent Belief}
\label{sec:cstbel}

Our analysis shows that in the intuitionistic propositional setting, knowledge and provably consistent belief are axiomatized by the same logical system, \iel. 

This situation is quite different from what we observe in classical epistemic logic. 
Indeed, in the classical setting there is a variety of systems for consistent belief: \textsf{D}, \textsf{KD4}, \textsf{KD45}, and systems for knowledge: \textsf{T}, \textsf{S4}, \textsf{S5}, that reflect different shades of belief and knowledge. 
However, similar axiom systems based on intuitionistic logic with the co-reflection principle \[ A\imp \bk A, \]
are all equivalent to \iel.\footnote{Given that the classical truth condition in \textsf{T}, \textsf{S4}, and \textsf{S5} is formulated by intuitionistic reflection $\bk A\imp \neg\neg A$, rather than classical reflection $\bk A\imp A$.} 
Does this mean that intuitionistic knowledge is just provably consistent belief? 
Not necessarily. However, it does mean that the basic intuitionistic epistemic logic \iel\ does not distinguish intuitionistic knowledge from intuitionistic provably consistent belief, just like the classical epistemic logic {\sf S5} does not distinguish knowledge from true belief.

\subsection{$\bk$ as $\neg\neg$}

\cite{Dosen1984} proposes an intuitionistic modal logic, \sft{Hdn$\bx$}, in which $\bx$ is read as intuitionistic $\neg\neg$, i.e.\ 
$$\bx A \bimp \neg\neg A.$$ 
\sft{Hdn$\bx$} validates $A \imp \bx A$ and invalidates $\bx A \imp A$. 
Could Do\v{s}en's $\bx$ be an intuitionistic epistemic operator?

We argue not. If it were it would follow that all classical theorems are known intuitionistically. 
By Glivenko's Theorem,  if ${\sf CPC} \proves A$ then \sft{Hdn$\bx$} $\proves \bx A$. 
Such a $\bx$ is not intuitionistic knowledge but rather a simulation of classical knowledge within \ipc.

Technically speaking, Do\v{s}en's modality $\neg\neg$ is strictly weaker than $\bk$: \iel\ proves $\bk A\imp\neg\neg A$ whereas $\neg\neg A\imp \bk A$ is not valid (e.g.\ when $A$ is the law of excluded middle
\footnote{Hence the classical truth of a proposition does not imply that it is verified.}
). 
Furthermore, $\sft{Hdn$\bx$}\proves \bx(X \imp Y) \bimp (\bx X \imp \bx Y)$ but neither of our systems have $\bk(X \imp Y) \bimp (\bk X \imp \bk Y)$. 

As a formal logical system,  $\sft{Hdn}\bx$ strictly extends $\iel$.

\subsection{Provability Semantics}
\label{sec:embed}

G\"odel, in \cite{Godel1933}, offered a provability semantics for intuitionistic logic via a syntactical embedding of \sft{IPC} into the classical modal logic \siv, which he considered a calculus for classical provability.
By extending \siv\ with a verification modality, $\bv$, and specifying an appropriate translation, we can explain each of our systems \ielm\ and \iel\ in the way G\"odel explained intuitionistic logic by interpreting them in the logic of provability \siv. 

Let \sivv$^-$ be the classical bi-modal logic with the axioms and rules of \siv\ for $\bx$, and the axioms and rules of modal logic \sft{K} for $\bf V$, along with the additional axiom $\bx A \imp {\bf V} A$. 
Let \sivv\ be \sivv$^-$ + $\neg \Box\bv \bot$.
By design, \sivv$^-$ may be regarded as the basic logic of verification, and \sivv\ as the basic logic of consistent verification.
The G\"odel translation \[ \mbox{\emph{$tr(A)\ =$ `box every subformula of $A$'}} \]  yields an embedding of $\mathcal{L} \in \{\ielm, \iel \}$ into $\mathcal{L}_\bx \in \{\sivv^{-}, \sivv \}$: 
\[ \mathcal{L} \proves A \ \Rightarrow \ \mathcal{L}_\bx \proves tr(A).
\footnote{It follows from the results of \cite{Artemov2014,Protopopescu2015} that this embedding is faithful for $\mathcal{L} = \iel^-, \iel$.
Note that there a stronger version of \sivv\ was given with $\neg \bv \bot$ instead of $\neg\Box\bv\bot$; $\neg\bx\bv\bot$ enables an extension of the arithmetical semantics for the Logic of Proofs \cite{Artemov2001,Artemov2005a} to be given for \iel, , see \cite{Protopopescu2016arxiv}.}
\]

The G\"odel translation interprets $\bk p$ as $\Box{\bf V}\Box p$, hence the verification of $p$ in \sivv\ is rather the checking of the provability of $p$. 
The ideology behind our systems also allows a more direct reading of verification under which $\bk p$ constructively verifies $p$ itself rather than `{\em $p$ is provable}', (\ref{second}) above. 
This reading can be captured by the translation of $\bk p$ as $\Box{\bf V}p$ which can be handled by the bimodal logic of constructive verification, for $\mathcal{L}_\bx \in \{ \sivv^{-}, \sivv \}$ \[\mathcal{L}_{\bx}^\prime = \mathcal{L}_\bx + ({\bf V}p\imp{\bf V}\Box p).\] 
We leave this line of research to future studies (see \cite{Protopopescu2015}). 

\section{Models for Intuitionistic Epistemic Logic}\label{sec:mod}

\begin{definition}[\ielm-model]\label{df:ielmmod}
A model for \ielm\ is a quadruple $\langle W, R, \forces, E\rangle$ such that 

\begin{enumerate}
 \item $\langle W,R,\forces\rangle$ is an intuitionistic model: $\langle W,R\rangle$ is a non-empty partial order ($R$ is a `cognition' binary relation on $W$), $\forces$ is a monotonic evaluation of propositional letters in $W$; 
 \item $E$ is a binary `knowledge' relation on $W$ coordinated with the `cognition' relation $R$: 
  \begin{itemize}

    \item $E(u) \subseteq R(u)$ for any state $u$;
    \footnote{Let $R(u)$ and $E(u)$ denote the $R$-successors and the $E$-successors, respectively, of some state $u$.}

    \item $uRv$ yields $E(v) \subseteq E(u)$;

  \end{itemize}

\item $\forces$ is extended to epistemic assertions as
\begin{itemize}
 \item $u\forces \bk A\ \ $ iff $\ \ v \forces A$ for all $v \in E(u)$.
\end{itemize}
\end{enumerate}

A formula $A$ is true in a model, if $A$ holds at each world of this model.
$\ielm\forces A$, or $\forces A$ for short, means that $A$ holds in each \ielm-model. 

\end{definition}

In Kripke model-theoretic terms the intuitionistic truth of $A$ is represented as the impossibility of a situation in which $A$ does not hold. 
To represent $\bk$ in the same model-theoretic terms we suggest the following: in a given world $u$, there is an `audit' set of possible worlds $E(u)$, the set of states $E$-accessible from $u$, in which verifications could possibly occur. 
An $R$-successor of a state $u$ can be thought of as an `in principle (logically) possible' cognition state given $u$, and an $E$-successor can be thought of as a `possible' state of verification.  
Belief, hence, is `truth in any audit world', i.e.\ no matter when and how an audit occurs, it should confirm $A$. 

Note that $E(u)$ does not necessarily contain $u$, hence the truth of $\bk A$ at $u$ does not guarantee that $A$ holds at $u$. 
Therefore, $\bk A \imp A$ does not necessarily hold. 
In the extreme $E(u)$ can coincide with $R(u)$, in which case $\bk A \imp A$ would hold.
Furthermore, the condition $E\subseteq R$ coupled with the monotonicity of truth w.r.t.\ $R$ ensures the validity of $A \imp \bk A$.

As for intuitionistic logic, we can think of \ielm-models as representing the states of information of an ideal researcher. 
Audit sets are monotone with respect to intuitionistic accessibility $R$. 
This corresponds to the Kripkean ideology that $R$ denotes the discovery process, and that things become more and more certain in the process of discovery. 
As the set of intuitionistic possibilities, $R(u)$, shrink, audit sets, $E(u)$, shrink as well. 
The monotonicity of truth represents the idealization of the researcher's memory; once a proposition becomes true, its truth is retained forever.

\begin{definition}[\iel-model]\label{df:ielmod} An \iel-model is an \ielm-model as above with the additional condition that $E$ is non-empty:
\smallskip

\begin{enumerate}
\item[\textup{4.}] $E(u)\neq\emptyset$ for each $u\in W$.
\end{enumerate}

\end{definition}
That audits are consistent is reflected in condition (4): in \iel-models $\neg \bk \bot$ holds. 
Indeed, for each world $u$, $E(u)\neq\emptyset$, then there is a $v\in E(u)$. Since $v\not\forces\bot$, $u\not\forces \bk \bot$ for each $u$, hence $w\forces\neg\bk \bot$  for each $w$. 

Again, note that $E(u)$ need not contain $u$, hence reflection is not guaranteed to hold. 

Note that the truth of $\bk A$ at $u$ in an \iel-model does guarantee that $A$ is true at some $v \in R(u)$, since $E(u) \subseteq R(u)$ and $E(u)\neq\emptyset$, this guarantees $\neg\neg A$ is true at $u$ also (cf. Theorem~\ref{th:sound}); this illustrates our earlier comment that $\bk A$ establishes the possibility of the intuitionistic truth of $A$.

In the limit case where $R(u)=\{u\}$, the audit set $E(u)$ is also $\{u\}$, and hence coincides with $R(u)$, i.e.\ `leaf nodes' are $E$-reflexive.  
Note that at `leaf nodes', intuitionistic evaluation behaves classically; at such a $u$, $u \forces \bk A \imp A$ for all $A$'s. In the epistemic case, at `leaf worlds' the classical factivity of $\bk$ holds.

\begin{lemma}[Monotonicity]\label{lm:mon}
 For each model and a formula $A$, if $u \forces A$ and $uRv$ then $v \forces A$.
\end{lemma}

\begin{proof}
It suffices to check \ielm-models. 
Monotonicity holds for the propositional connectives, we show this just for $\bk$. 
Assume $u \forces \bk A$, then $x \forces A$ for each $x \in E(u)$. Take an arbitrary $v$ such that $uRv$ and arbitrary $w\in E(v)$, hence $w \in E(u)$. 
Therefore, $w \forces A$ and hence $v \forces \bk A$.

\end{proof}

\begin{theorem}[Soundness]\label{th:sound}
 For $\mathcal{L}\in\{\ielm,\iel \}$, if $\mathcal{L}\proves A$ then $\mathcal{L} \forces A$.
\end{theorem}
\begin{proof} 
By induction on derivations in \sft{IEL}. We check the epistemic clauses only.
\medskip\par
1) $\bk(A \imp B) \imp (\bk A \imp \bk B)$ for \ielm-models. It suffices to check that $u \forces \bk(A \imp B)$ and $u \forces \bk A$ yield $u \forces \bk B$. Assume $u \forces \bk(A \imp B)$ and $u \forces \bk A$, then for all $v\in E(u)$, $v \forces A \imp B$ and $v \forces A$, hence $v \forces B$. By definition, this means that $u \forces \bk B$. 
\medskip\par
2) $A \imp \bk A$ for \ielm-models. Assume $u \forces A$. By monotonicity, for all $v\in R(u)$, $v \forces A$. Since $E(u) \subseteq R(u)$, for any $w\in E(u)$, $w \forces A$, but then $u \forces \bk A$ as well.
 \medskip\par 
3) $\bk A \imp \neg \neg A$ for \iel-models. 
Assume $u \forces \bk A$. By monotonicity, for each $v\in R(u)$, $v\forces \bk A$ as well. Pick an arbitrary $v\in R(u)$; it suffices now to show that $v\not\forces\neg A$. Since $E(v)\neq\emptyset$, there is $w\in E(v)\subseteq E(u)\subseteq R(u)$, and, by definitions, $w\forces A$. This yields that $v\not\forces\neg A$, hence $u\forces\neg\neg A$.
\end{proof}

\begin{theorem}\label{th:strictincl}
$\ielm \subset \iel$.
\end{theorem}
\begin{proof} \
$\ielm \neq \iel$. Consider the following \ielm-model $\mathcal{M}_1$: $W$ is a singleton, $R$ is reflexive and $E$ is empty. 

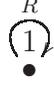
\begin{figure}[H]
\centering\mbox{

\begin{xy}
(-15,4)*{1};
(-15,0)*+{\bullet}="B"; 
{\ar@{->}@(ul,ur)^R "B";"B"};

\end{xy}

}
\caption{\ielm-model $\mathcal{M}_1$}
\end{figure}
\noindent
Since $E(1)=\emptyset$, $1 \forces \bk A$, but since $1 \nforces A$, $1 \forces \neg A$ hence $1 \nforces \neg \neg A$.

\end{proof}

\begin{theorem}[Completeness]
For $\mathcal{L}\in\{ \ielm,\iel \}$, if $\mathcal{L}\forces A$ then $\mathcal{L} \proves A$.
\end{theorem}
\begin{proof}
See Appendix \ref{sec:comp}.
\end{proof}
\begin{theorem}\label{th:nfk}
For $\mathcal{L}\in\{\ielm,\iel \}$,
 $$\mathcal{L} \nvdash \bk A \imp A.$$
\end{theorem}
\begin{proof}
Consider the following model: $1R2$, $R$ is reflexive (and vacuously transitive), $E(1) = E(2) =  \{ 2 \}$, $p$ is atomic and $2 \forces p$.
\medskip\par
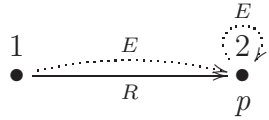
\begin{figure}[H]
\centering\mbox{
\begin{xy}
(-15,4)*{1};
(15,4)*{2};
(15,-4)*{\emph{p}};
(-15,0)*+{\bullet}="B"; 
(15,0)*+{\bullet}="C"; 
{\ar_R "B";"C"};
{\ar@/^/^E@{.>} "B";"C"};
{\ar@{.>}@(ul,ur)^E "C";"C"};
\end{xy}
}
\caption{\iel-model $\mathcal{M}_2$}
\end{figure}
\noindent
Clearly, $1 \forces \bk p$ and $1 \nforces p$, hence $$1\not\forces\bk p\imp p.$$
\end{proof}

Model $\mathcal{M}_2$ exemplifies the point  that intuitionistic verification guarantees the possibility of intuitionistic truth (see section \ref{sec:knf}): 
$$1\forces\bk p\imp \neg\neg p.$$

In the logics of intuitionistic knowledge, though reflection does not hold generally, it does hold for negated formulas.

\begin{theorem}\label{th:fknref} 
$\iel \proves \bk \neg A \imp \neg A.$\footnote{It is easy to check that $\bk \neg A \imp \neg A$ could be used instead of $\bk A\imp \neg\neg A$ to axiomatize \iel.}
\end{theorem} 

\begin{proof} \ 

1. $\bk \neg A \imp \neg\neg\neg A$ - intuitionistic reflection;

2.  $\bk \neg A \imp \neg A$ - from 1 by $\neg X \bimp \neg\neg\neg X$.

\end{proof}

Intuitionistic knowledge and negation commute: the impossibility of verifying $A$ is equivalent to verifying that $A$ cannot possibly hold. 

\begin{theorem}\label{th:kcomm}
$\iel \proves \neg \bk A \bimp \bk \neg A.$
\end{theorem}
\begin{proof}
  `$\leftarrow$' follows by Theorem \ref{th:fknref} and Theorem \ref{th:ielth} part 3. Let us check `$\imp$':
\medskip\par
1.  $A \imp \bk A$ - co-reflection;

2. $\neg\bk A \imp \neg A$ - contrapositive of 1;

3. $\neg A \imp \bk \neg A$ - co-reflection;

4. $\neg\bk A \imp \bk \neg A$ - from 2 and 3.
\end{proof}

In logics of intuitionistic knowledge, the impossibility of verification is equivalent to the impossibility of proof, see section \ref{sec:percival} for discussion.

\begin{theorem}\label{th:nka=na}
$\iel\ \proves \neg \bk A \bimp \neg A.$
\end{theorem}
\begin{proof}
 `$\rightarrow$' is shown in Theorem \ref{th:kcomm}, line 2.  Let us check `$\leftarrow$':

  1. $\neg A$ - assumption; 

  2. $\bk \neg A$ - from 1 and co-reflection;

  3. $\neg \bk A$ - from 2 and Theorem \ref{th:kcomm}.
\end{proof}

Within the intuitionistic knowledge framework, no truth is unverifiable, see section \ref{sec:percival} for discussion.

\begin{theorem}\label{th:vfbl}
$\iel \proves \neg(\neg \bk A \land \neg \bk \neg A).$
\end{theorem}
\begin{proof} $\ $

  1. $\neg \bk A \land \neg \bk \neg A$ - assumption;

  2. $\bk \neg A \land \bk \neg \neg A$ - by Theorem \ref{th:kcomm};

  3. $\neg A \land \neg\neg A$ - by Theorem \ref{th:fknref}; 

  4. $\bot$ - from 3;

  5. $\neg(\neg \bk A \land \neg \bk \neg A)$ - from 1--4.
\end{proof}

Intuitionistic verifications do not have the disjunction property. 

\begin{theorem}\label{th:nkdisj}
For $\mathcal{L}\in\{\ielm, \iel \}$, \[\mathcal{L}\nproves \bk(A \lor B) \imp (\bk A \lor \bk B).\]
\end{theorem}
\begin{proof}
Consider the following model. $1R2, 1R3$ ($R$ is reflexive); $1E2, 1E3, 2E2, 3E3$; $p$ is atomic and $3 \forces p$. 
Since $2,3 \forces p \lor \neg p$, $1
\forces \bk(p \lor \neg p)$.
However, $1 \nforces \bk p$, and $1 \nforces \bk \neg p$.

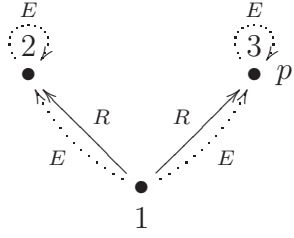
\begin{figure}[H]
\centering\mbox{
\begin{xy}
(0,19)*{\mbox{3}};
(-15,-4)*{\mbox{1}};
(-30,19)*{\mbox{2}};
(4,15)*{\mbox{$p$}};
(-30,15)*+{\bullet}="2";
(0,15)*+{\bullet}="3";
(-15,0)*+{\bullet}="1";
{\ar_R "1";"2"};
{\ar^R "1";"3"};
{\ar@/_/_E@{.>} "1";"3"};
{\ar@{.>}@(ul,ur)^E "3";"3"};
{\ar@/^/^E@{.>} "1";"2"};
{\ar@{.>}@(ul,ur)^E "2";"2"};
\end{xy}
}
\caption{\iel-model $\mathcal{M}_3$}
\end{figure}
\end{proof}

\begin{theorem}\label{th:fkrule}
For $\mathcal{L}\in\{\ielm, \iel \}$, the reflection rule $\vdash \bk A\ \Rightarrow\ \vdash A \text{ is admissible in }\mathcal{L}.$
\end{theorem}
\begin{proof}
  Suppose $\nproves A$, hence, by completeness, there is an $\mathcal L$-model $\mathcal{M}=\langle W, R, \forces, E \rangle$ with a node $x\in W$ s.t. $x \nforces A$. Construct a new $\mathcal L$-model,  $\mathcal{N}=\langle W^\prime, R^\prime, \forces^\prime, E^\prime\rangle$ such that 

 \begin{itemize}
\item $W^\prime=W\cup\{x_0\}$ ($x_0$ is a new node);
 \item $x_0 R^\prime u$ and $x_0 E^\prime u$ for all $u\in W^\prime$, $R^\prime$ coincides with $R$ and $E^\prime$ coincides with $E$ on $W$;
\item $x_0\not\forces^\prime p$ for each atomic sentence $p$ and $\forces^\prime$ coincides with $\forces$ on $W$.
\end{itemize}
Clearly $\mathcal{M}$ is a generated submodel of $\mathcal{N}$, hence $\forces^\prime A$ coincides with $\forces A$ on $W$ for all $A$. 
Furthermore, $x_0 \nforces^\prime \bk A$, since $x \nforces^\prime A$ and $x_0E^\prime x$. Therefore, ${\mathcal L}\nproves \bk A$ also. 
\end{proof}

\medskip\par

\begin{theorem}[Disjunction Property]
For $\mathcal{L}\in\{\ielm, \iel \}$, \begin{center}if $\mathcal{L} \proves A \lor B$ then either $\mathcal{L} \proves A$ or $\mathcal{L}\proves B$.\end{center}
\end{theorem}

\begin{proof}
  Assume $\nproves A$ and $\nproves B$. By completeness, $\nforces A$ and $\nforces B$. Hence there are $\mathcal L$-models $\mathcal{M}_1 = \langle W_1, R_1, \forces_1, E_1\rangle $ and $\mathcal{M}_2 =  \langle W_2, R_2, \forces_2, E_2\rangle$ with nodes $x_1 \in W_1$ and $x_2 \in W_2$ such that $x_1 \nforces_1 A$ and $x_2 \nforces_2 B$. We define a new $\mathcal L$-model $\mathcal{M} = \langle W, R, \forces, E \rangle$ such that 

\begin{itemize}
 \item $W = W_1 \cup W_2 \cup \{ x_0 \} $ where $x_0 \notin W_1$ and $x_0 \notin W_2$ ($W_1$ and $W_2$ are assumed disjoint).

 \item $x_0Ru$ and $x_0Eu$ for all $u \in W$, $R$ coincides with $R_i$ on $W_i$, and $E$ coincides with $E_i$ on $W_i$, $i=1,2$. 

 \item $x_0 \nforces p$ for each atomic sentence $p$, $\forces$ coincides with $\forces_i$ on $W_i$, $i=1,2$.
\end{itemize}
It is easy to check that for each $i=1,2$ and each $x\in W_i$, 
\[ x\forces A\ \ \mbox{\it iff}\ \ \ x\forces_i A .
\]
We claim that $x_0 \nforces A \lor B$, hence ${\mathcal L}\nproves A \lor B$.
Indeed, if $x_0 \forces A \lor B$, then $x_0 \forces A$ or $x_0 \forces B$. If $x_0 \forces A$ then, by monotonicity,  $x_1 \forces A$, hence $x_1\forces_1 A$ which contradicts our assumptions.  Case $x_0 \forces B$ is symmetric. 

\end{proof}

Despite Theorem \ref{th:nkdisj}, intuitionistic epistemic logic has a weak disjunction property for verifications. 

\begin{corollary}
For $\mathcal{L}\in\{\ielm, \iel \}$, \begin{center} if $\mathcal{L}\proves \bk(A \lor B)$ then either $\mathcal{L}\proves \bk A$ or $\mathcal{L}\proves \bk B$.\end{center}
\end{corollary}
\begin{proof}
  Assume $\mathcal{L} \proves \bk(A \lor B)$ then, by Theorem \ref{th:fkrule}, $\mathcal{L}\proves A \lor B$, hence $\mathcal{L}\proves A$ or $\mathcal{L}\proves B$. In which case $\mathcal{L}\proves \bk A$ or $\mathcal{L}\proves \bk B$ by co-reflection.
  
\end{proof}

\subsection{Modeling Knowledge vs.\ Belief}\label{sec:kvsbel}

As an illustration of our informal remarks, in \ref{sec:knf}, on the difference between intuitionistic belief and knowledge consider again the example of `all swans are white'. 
The following \ielm\ model, $\mathcal{M}_4$, seems to model fairly the belief of an agent before the European discovery of Australia.
$1R2, 1R3$ ($R$ is reflexive); $1E3$; $p$ is `all swans are white' and $3 \forces p$.

\begin{figure}[H]
\centering\mbox{
\begin{xy}
(0,19)*{\mbox{3}};
(-15,-4)*{\mbox{1}};
(-30,19)*{\mbox{2}};
(4,15)*{\mbox{$p$}};
(-30,15)*+{\bullet}="2";
(0,15)*+{\bullet}="3";
(-15,0)*+{\bullet}="1";
{\ar_R "1";"2"};
{\ar^R "1";"3"};
{\ar@/_/_E@{.>} "1";"3"};
\end{xy}
}
\caption{\ielm-model $\mathcal{M}_4$}
\end{figure}

The underlying intuitionistic model represents the logical possibilites of developing the agent's information regarding the truth of $p$. 
The epistemic part of the model represents the verifications the agent has performed.
In this case all verifications confirm $p$, hence $\bk p$ holds at 1.
However this is a mere belief: the truth condition fails because $1 \nforces \neg\neg A$. 
This models the historical situation in which it was considered ``known'' that all swans are white, but which was in fact only a
belief because the situation in which $p$ does not hold was not considered epistemically possible.

By contrast, consider the \iel\ model $\mathcal{M}_3$ from Theorem \ref{th:nkdisj}.
This has the same logical possibilities, but the agent has verified each of them.
In this case $\bk p$ does not hold at 1; $p$ is not known because there is verification that it can be false. 

\section{\iel\ and Intuitionistic Responses to the Knowability Paradox}\label{sec:kbpdx}

The Church-Fitch `knowability paradox' is an informal interpretation of a classical derivation in bi-modal logic with the modalities $\bk$ and $\dm$ discovered by Church \cite{Church2009} and reported by Fitch \cite{Fitch1963}. 
The proof shows that
\[
\phantom{\text{verificationist knowability}} A \imp \kb A\tag{verificationist knowability}
\]
classically implies
\[
 A \imp \bk A\tag{omniscience}.
\]

\begin{proposition}[Church-Fitch]\label{th:CF1} 
Verificationist knowability as a schema classically yields omniscience.
\end{proposition}

\begin{proof} $\ $

1. $(\mathit{p\wedge\neg\bk p})\rightarrow\Diamond\mathbf{ K}(\mathit{ p\wedge\neg\bk p})$ - verificationist knowability;

2. $\mathbf{K}(\mathit{ p\wedge\neg\bk p})$ - assumption;

3. $\bk p\wedge \bk\neg\bk p$ - from 2 by standard modal reasoning;

4. $\bk p\wedge \neg\bk p$ - from 3 and reflection;

5. $\neg \mathbf{K}(\mathit{ p\wedge\neg\bk p})$ - from 2--4;

6. $\Box \neg \mathbf{K}(\mathit{ p\wedge\neg\bk p})$ - from 5 and necessitation;

7. $\neg\Diamond \mathbf{K}(\mathit{ p\wedge\neg\bk p})$ - from 6 and $\Box\neg X \rightarrow \neg\Diamond X$; 

8. $\neg(\mathit{p\wedge\neg\bk p})$ - from 1 and 7; 

9. $p\rightarrow\neg\neg\bk p$ - from 8 and $\neg(X\wedge Y)\rightarrow(X\rightarrow\neg Y)$; 

10. $p\rightarrow\bk p$ - from 9 and double negation elimination.

\end{proof}

The informal interpretation of this formal proof takes it to show that
\begin{center}
 \emph{all truths are knowable}
\end{center}
implies
\begin{center}
 \emph{all truths are known.}
\end{center}
That all truths are knowable is taken to be definitive of intuitionistic truth, hence the `knowability paradox' appears to be a \emph{reductio ad absurdum} of the very idea of intuitionism.
\footnote{We use the term `knowability paradox' to denote the informal argument, based on the proof, which is supposed to refute intuitionism. The term `Church-Fitch proof' denotes the derivation itself.}

Intuitionistic responses to the paradox so far have accepted this informal interpretation of the Church-Fitch construction, and hence been committed to showing that an intuitionistic conception of truth and knowledge does not yield co-reflection, and hence is not committed to omniscience, e.g.\ \cite{Proietti2012, Dummett2001, Dummett2009, Tennant1997, Tennant2009, Williamson1982, Williamson1988, Williamson1992, Williamson1994}. 
We argue that the proper intuitionistic response is simply that there is no paradox;\footnote{See e.g.\ \cite{Dummett2009, Usberti1994, Khlentzos2004, Rasmussen2009}.} intuitionistically the `knowability paradox' is a pseudo-problem which holds \emph{only} from a classical standpoint.
The supposedly devastating conclusions of the knowability paradox are solely the product of a classical reading of the principles of the Church-Fitch construction, combined with an incorrect representation of intuitionistic truth in a classical framework. 

\subsection{\iel\ Response to the Knowability Paradox}
\label{sec:ielkbpdx}

Construed as a problem for intuitionism the `knowability paradox' depends on the following assumptions:

\begin{enumerate}[noitemsep]
 \item\label{it:clct} $A \imp \bk A$ means \emph{all truths are known}.
 
 \item\label{it:clvk} $A \imp \kb A$ means \emph{all truths are knowable}.

 \item\label{it:int=kb} That all truths are knowable is definitive of intuitionistic truth.
\end{enumerate}

Since the `knowability paradox' incorporates a particular view about the nature of intuitionistic truth and its relationship to knowledge, we restrict our discussion to the perspective afforded to us by our system of intuitionistic knowledge, \iel\ (though much of what we say holds just as well for intuitionistic belief).
From this perspective we argue none of these hold. 

\subsubsection{Reply to \ref{it:clct}}\label{sec:rep1}

$A \imp \bk A$ can be understood as claiming \emph{all truths are known} only on a classical reading. 
Intuitionistically it means something very different, namely, \emph{constructive truth, i.e.\, proof, yields verification/knowledge}, which is central to an intuitionistic view of knowledge.

\subsubsection{Reply to \ref{it:clvk}}\label{sec:rep2}

Knowability is taken to be a definitive characteristic of intuitionistic truth, but $A \imp \kb A$ is not a good formalization of this idea in classical logic.

As a classical principle it does not capture the intended (intuitionistic) relation between truth and knowledge (see \cite{Artemov2010, Artemov2012}).
The straightforward classical logic reading of $A \imp \dm\bk A$ says \emph{all classical truths are knowable}, which is plainly false.
To be an adequate classical formalization of an intuitionistic notion of knowability one has to `build in' the constructivity of intuitionistic truth; when this is done\footnote{In the principles SK and MK of \cite{Artemov2010, Artemov2012}.} no paradoxical conclusions follow.

As an intuitionistic principle it is not immediately clear what $A \imp \dm\bk A$ says. 
This may be construed as saying that `all (intuitionistic) truths are know\-able\-/verifiable', but it is not clear what theoretical or expressive advantage is gained by adopting this, given the stronger principle of co-reflection $A\imp\bk A$ holds intuitionistically. Perhaps, a combination of reading BHK-proofs as timeless platonic entities not necessarily available to the knower with a strict reading of $\bk A$ as {\em $A$ is actually known} (cf.\ section~\ref{sec:bhk-meaning}) could provide a reasonable semantics 
for the knowability principle $A \imp \kb A$: \[ \mbox{\em if a proposition has a proof then it is possible to actually know it.} \]

\subsubsection{Reply to \ref{it:int=kb}}\label{sec:rep3}

The characteristic feature of intuitionistic truth is its constructivity; a proposition is true if there is a proof of it. 
A proof is an especially strict kind of verification; hence intuitionistic truth yields verification, hence knowledge. 
But it is precisely for this reason that no intuitionistic truth is beyond the possibility of knowledge; intuitionistic truth is knowable \emph{because it is constructive}.
Hence constructivity, not knowability, is the definitive feature of intuitionistic truth.

Moreover, we argue that $A \imp \bk A$ is the intuitionistic formal expression of the constructivity of intuitionistic truth. Knowability in its different possible forms appears to be its immediate consequence. 
For example, $A \imp \bk A$ implies each of $\neg(\neg \bk A \land \neg \bk \neg A)$, $A \imp \neg\neg\bk A$ (see \ref{sec:intkb}), and assuming `what is proved is possible' $A \imp \dm\bk A$.
Each of these may be regarded as possible intuitionistic formalizations of `all truths are knowable'. All are easy consequences of the constructivity of intuitionistic truth.

\medskip\par

For these reasons, the `knowability paradox' does not constitute a problem for the intuitionistic notion of truth.

\subsection{Intuitionistic Criticism of Church-Fitch}
\label{sec:intCF}

One line of intuitionistic response has been to criticize the Church-Fitch proof.
Williamson argues, \cite{Williamson1982}, that the proof is intuitionistically propositionally invalid, since it involves a step of double negation elimination.
Intuitionistically Church-Fitch establishes only
\[
\phantom{\text{intuitionistic knowability}} A \imp \neg\neg \bk A\tag{intuitionistic knowability}
\]
which, read intuitionistically, is not paradoxical, hence the intuitionist is not committed to co-reflection.
\footnote{Note that though reflection is used in the proof, line 4 of Proposition \ref{th:CF1}, it is of the intuitionistically acceptable kind $\bk \neg A \imp \neg A$, Theorem \ref{th:nfk}.} 

Note, however, that the conclusions of the Church-Fitch proof in either classical or intuitionistic form, $A\imp \bk A$ or $A\imp \neg\neg\bk A$, are valid in \iel; neither depend on verificationist knowability or on the Church-Fitch proof.
As a derivation, the Church-Fitch proof turns out to be irrelevant to intuitionistic foundations; arguments against intuitionistic truth and knowledge have to take a different approach.

\subsection{Intuitionistic Knowability}
\label{sec:intkb}

Another intuitionistic approach to showing that co-reflection/omniscience does not hold argues that intuitionistic knowability itself is a better formalisation than verificationist knowability of the informal idea that all truths are knowable, or of the feature of the relation between intuitionistic truth and knowledge which this is intended to express (see \cite{Dummett2009, Rasmussen2009, Devidi2001}). 

There is much to recommend this, though, following our comments in \ref{sec:rep2} and \ref{sec:rep3}, we would argue that it, like verificationist knowability, does not capture the basic relation between intuitionistic truth and knowledge.
Intuitionistic knowability can be read as a formalization of `all truths are knowable'; Dummett \cite[p.52]{Dummett2009} reads it as ``if $A$ is true then the possibility that $A$ will come to be known always remains open''. 
It can also be read as a form of weak co-reflection -- indeed one might consider weak versions of our intuitionistic epistemic systems formulated on its basis
(see section \ref{sec:percival} for further discussion of $A \imp \neg\neg\bk A$) -- but it too, like verificationist knowability, does not fully capture the relation between intuitionistic truth and knowledge. 

\section{Criticisms of \iel\ Principles}\label{sec:crit}

\subsection{Intuitionistic Rejection of Co-reflection}
\label{sec:reject}

\subsubsection{Hart}\label{sec:hart}

The common feature of intuitionistic responses to the `knowability paradox' has been the commitment to rejecting co-reflection, based on accepting its classical reading as omniscience.
From the very first intuitionistic response to the knowability paradox, we find this rejection even in the face of direct, intuitionistically acceptable, arguments for the validity of co-reflection.

Hart's \cite[p.165]{Hart1979} sets the pattern -- though he thinks such a response is mistaken. 

    \begin{quote}\small{
    Incidentally, on an intuitionist reading, it just might be that every truth is known. For being in an intuitionist position to assert that $\forall x (Fx \imp Gx)$ requires a method which given an object and a proof that it is $F$, yields a proof that it is $G$. In the present instance this means: suppose we are given a sentence \dots and a proof that it is true. Read the proof; thereby you come to know that the sentence is true. Reflecting on your recent learning, you recognize that the sentence is now known by you; this shows that the truth is known. If this argument is intuitionistically acceptable \dots then I think that fact reflects poorly on intuitionism; surely we have good inductive grounds for believing that there are truths as yet unknown.}
    \end{quote}

This has all the elements of a justification of intuitionistic co-reflection.
However Hart does not apply the argument to the reading of co-reflection, instead treating his argument that `proof yields knowledge' as a justification for `all truths are known'. 
This is unstable, one can reject intuitionism altogether and the argument for the validity of co-reflection with it, of course, but one cannot argue that a classical understanding of a principle invalidates an intuitionistic reading of it. 

\subsubsection{Williamson}\label{sec:williamson}

Williamson, in \cite{Williamson1982, Williamson1988, Williamson1994}, considers Hart's argument supporting $A \imp \bk A$, but rejects it for the same reason, and seeks to devise a form of intuitionistic semantics which invalidates co-reflection. 

To do this Williamson \cite{Williamson1988} distinguishes between \emph{proof-tokens} and \emph{proof-types}. 
Proof-tokens are of the same type just if they have the same structure and conclusion, though they may be effected at different times. 
Co-reflection holds for proof-tokens, any proof-token of $A$ can be turned into a proof-token of $\bk A$, but not for proof-types. 
In the case of proof-types $A \imp \bk A$ says that there is a function which takes a proof-type of $A$ to a proof-type of $\bk A$.
In this context `$\bk A$' is read as `there exists a time $t$ such that $A$ will have been proved at $t$'. 
Moreover, the validity of co-reflection requires that this function be \emph{unitype}, meaning that if inputs, $p$ and $q$,  are of the same type then the outputs, $f(p)$ and $f(q)$, are of the same type also. 
Hence ``a proof of $A \imp \bk A$ is a unitype function that evidently takes any proof token of $A$ to a proof token, for some time $t$, of the proposition that $A$ is proved at $t$'', \cite[p.430]{Williamson1988}. 
Williamson's contention is that such a function does not exist in all cases. 
It exists where we already have an input for the function, i.e.\ a proof of $A$. 
In the case where we do not have an input all we can consider is the function $f$ itself, which takes us from hypothetical proof-tokens of $A$ to proof-tokens of $\bk A$. 
But such a function is not unitype. 
Assume that $p$ and $q$ are token-proofs of $A$ of the same type carried out at different times, then $f(p)$ and $f(q)$ will be proof tokens of different types. 
$f(p)$ is a proof that $\bk A$ is proved at time $t$ and $f(q)$ is a proof that $\bk A$ is proved at time $t^\prime$. 
Hence co-reflection is not generally valid.
\footnote{See \cite{Usberti1994} for an argument that such a function does exist; $f$ cannot operate on hypothetical proof tokens, since they do not exist; so $f$ can still be defined as a unitype function taking a proof of $A$ and returning a proof of $\bk A$. For an objection to this see \cite{Murzi2010}. On the debate about the status of hypothetical reasoning in intuitionism see \cite[p.30]{vanDalen1988v1} and the references contained therein.}

In response, we point out that the BHK semantics has no temporal com\-pon\-ent. 
Wil\-liamson interprets co-reflection as a kind of universal proof-checking (as do we, see section \ref{sec:pyk}), but the time at which the proposition was proved is, normally, not essential to checking a proof's correctness. 
Co-reflection asserts correctly that given a proof, $x$, of $A$ proof-checking produces another proof, $y$, that there exists a verification of $A$, namely a proof $x$ of $A$. 
As we see, co-reflection holds independently of the time $x$ is carried out or of whether it has already been constructed or is only hypothetical: proof-checking is a correct procedure the possibility of which is independent of any assumptions about specific proofs.
Williamson interprets co-reflection in a rather non-standard way by adding an alien temporal component to devise a reading under which co-reflection could fail, but the intuitionist need not accept this temporal aspect.

This exhibits the same instability found in Hart's response, which is attributable to working in an intuitionistic context without achieving a complete liberation from a classical conception of knowledge. 

\subsection{Other Intuitionistic Epistemic Logics}
\label{sec:altiel}

\subsubsection{Williamson}
\label{sec:williamsoniel}

The commitment to a classical conception of knowledge is even more in evidence in William\-son's formulation of an intuitionistic modal epistemic logic \cite{Williamson1992}, a goal of which is to invalidate co-reflection, while at the same time reflection, $\bk A \imp A$, is endorsed explicitly.
This is striking since Williamson treats intuitionistic knowledge as a kind of verification, and acknowledges that verifications need not be proofs.
\footnote{Williamson's verifications are somewhat different from ours (see note \ref{ft:Wmson}), since $\bk A$ always expresses an empirical proposition -- regarding the contingency of $\bk$ see \ref{sec:percival}.} 
Williamson's intuitionistic epistemic logic is not based on the standard BHK semantics, but rather on the idea that ``intuitionistic truth consists in the possibility of verification'' \cite[63]{Williamson1992}. 
Hence Williamson's intuitionistic epistemic logic does not capture the sense of BHK-based knowledge which is our goal, while also importing classical epistemic assumptions into an intuitionistic context.

\subsubsection{Proietti}\label{sec:proietti}

A more recent development of the basic approach taken by DeVidi and Solomon \cite{Devidi2001} is found in Proietti \cite{Proietti2012}, who develops an intuitionistic epistemic logic based on a Kripkean semantics.

Proietti's basic assumptions follow the pattern for intuitionistic responses to the knowability paradox: that even under intuitionistic assumptions $A \imp \bk A$ is invalid. At the same time he assumes explicitly that $\bk A \imp A$ holds in the logic.

It should be noted, however, that Proietti is not trying to analyze Brouwer's original intuitionistic paradigm of `truth as provability'. Proietti's starting point is rather the later Kripke semantics of intuitionistic logic, which is not ideologically and technically faithful to the original intuitionistic foundations. The relationship between truth, proof and knowledge does not arise in a Kripkean semantic context, since proof is not part of the picture, but for this very reason the intuitionistic considerations in favor of co-reflection and against reflection cannot come up.
\footnote{Another system of intuitionistic epistemic logic, presented as extending the BHK semantics to knowledge, is given by Hirai \cite{Hirai2010,Hirai2010a}. 
Reflection is valid in this system, the rationale for this is its admissibility in classical epistemic logic.
Additionally the co-reflection principle is not considered at all. 
It is clear, however, that it is not Hirai's aim to give an intuitionistic analysis of knowledge, but rather to model asynchronous communication between computational processes.
Accordingly it appears that Hirai's system has the same classical bias as the systems already discussed, but since his aims are rather divergent from ours, we offer this only as an observation about the formalism.
} 

\subsection{Percival and the `Paradoxes of Intuitionistic Knowledge'}\label{sec:percival}

The principle $A \imp \neg\neg \bk A$, we have seen, has been taken either as an intuitionistic solution to the knowability paradox, or as definitive of the relation between intuitionistic truth and knowledge, see sections \ref{sec:intCF} and \ref{sec:intkb}. 
Percival objects that intuitionistically $A \imp \neg\neg \bk A$ yields paradoxical consequences. 
Since it is an easy consequence of co-reflection, and hence valid in \ielm, his arguments are an objection to the BHK view of knowledge and to the 
intuitionistic epistemic logic that results from it.

Percival argues that intuitionistically $A \imp \neg\neg \bk A$ implies $\neg\bk A \bimp \neg A$ and $\neg(\neg\bk A \land \neg\bk\neg A)$, both of which are intuitionistically unacceptable \cite[p.183]{Percival1990}. 
The first Percival reads as claiming that the falsehood of $A$ and ignorance of $A$ are logically equivalent intuitionistically. 
But, he argues, this cannot be. 
Assume that $\neg A$ is a mathematical proposition, hence necessarily true. 
Whether $A$ is not known, $\neg \bk A$, is a contingent matter. 
Hence there must be some state of some model where $\neg A$ holds and $\neg \bk A$ does not. 
The second Percival reads as claiming that no statement is forever undecided.
\footnote{This is also known as the `undecidedness paradox of knowability' \cite[]{Brogaard2009a}, since Percival's objections are to the intuitionistic argument against the knowability paradox discussed in section \ref{sec:intCF}.} 
He claims that this second consequence is just obviously false; there exists a $p$ for which $\neg\bk p \land \neg\bk\neg p$ holds. 

We argue that both Percival's `counterexamples' are valid epistemic principles within the BHK-based \iel\ paradigm (see Theorems \ref{th:nka=na} and \ref{th:vfbl}) and hence do not serve as decisive arguments against the incorporation of verification-based knowledge into an intuitionistic framework.
\footnote{Or against intuitionistic responses to the knowability paradox.} 

First, in terms of intuitionistic knowledge (\iel) $\neg\bk A \bimp \neg A$ claims proving $\neg A$ is equivalent to proving $\neg \bk A$. 
If we can show that a proof of $A$ reduces to a contradiction then $A$ cannot possibly hold, hence neither can $\bk A$. 
Conversely, if a proof of $\bk A$ reduces to a contradiction then there cannot be a verification of $A$, but every proof is also a verification, hence there cannot be a proof of $A$.\footnote{See \cite{Marton2006}.} 
The contingency or necessity of $\neg A$ and $\neg\bk A$ is not relevant because intuitionistically these are statements about the relationship between proofs and verifications of propositions -- whatever their modal status (cf. \cite[p.325]{Devidi2001}).
\footnote{Moreover, it is not clear this argument works in its own terms. If $\neg A$ is necessarily true, then $A$ is necessarily false, in which case $A$ \emph{cannot be known}, since for a necessarily false $A$ ignorance of $A$ is also necessary. Hence there is no state of any model where $\neg\bk A$ does not hold.} 

Second, in \iel, $\neg(\neg\bk A \land \neg\bk\neg A)$ claims that no truth is unverifiable, not that no truth remains forever undecided.
\footnote{Again see \cite{Devidi2001}, and section \ref{sec:rep3}.} 
Indeed, by the previous principle if $\neg \bk A$ held then so does $\neg A$, i.e.\ there is a proof of $\neg A$, hence $\neg A$ is verified and $\bk \neg A$. Once again, the contingency of some agent's ignorance is beside the point. 

Percival's conclusion that since ``[verificationist knowability] has consequences that are plainly unacceptable we \emph{do} know in advance that no intuitionistic defense \dots with a specific semantics \dots is going to work'' does not stand. 
The intended intuitionistic semantics, BHK, as extended to \iel, is such a semantics, and in its terms the consequences are quite acceptable. 

\subsection{\iel\ and Non-Mathematical Propositions}
\label{sec:ielnmp}

But there is a further, more general, premise that Percival assumes which he uses to support his conclusion, and which seems to apply to any attempt to enunciate an intuitionistic view of knowledge. 
He argues that ``as anti-realist sympathizers \dots admit, non-mathematical statements aren't susceptible to \emph{proof} and a proof-conditional interpretation of `$\imp$' isn't \emph{generally} viable. So an intuitionistic defense [against the `paradoxes of intuitionistic knowledge'] can't appeal to it'' \cite[183]{Percival1990}. 
According to this line of reasoning an intuitionistic view of knowledge cannot be put in terms of BHK, because BHK does not apply to all kinds of propositions. 
A legitimate intuitionistic defense must give a semantics that holds for all kinds of propositions, not just mathematical ones, and be ``independently plausible''. 

We respond that this is an illegitimate constraint on an intuitionistic view of knowledge. 
BHK is the intended semantics of intuitionistic logic, indeed the intuitionistic calculus was constructed to capture the BHK semantics not the other way around. 
There are, of course, many non-BHK semantics for \ipc, but it is acknowledged that they are more or less artificial, not true to the intentions of intuitionism, precisely because they do not represent the BHK view.\footnote{See \cite{Artemov2001,vanDalen2002}. 
``The intended interpretation of intuitionistic logic as presented by Heyting, Kreisel and others so far has proved to be rather elusive [\emph{authors' note:} but see \cite{Artemov2001}] \ldots however, ever since Heyting's formalisation, various, more or less artificial, semantics have been proposed'' \cite*[p.22]{vanDalen2002}. } 
To demand an intuitionistic semantics which rules out BHK is, to some extent, to demand an intuitionistic theory which is not intuitionistic, which is not legitimate.

The point of the objection is that the BHK interpretation cannot accommodate non-mathematical propositions. 
However this is plainly wrong since there is nothing specifically mathematical in the BHK description, it makes no mention of numbers, functions, sets, categories, types, etc.\ 
There are a variety of non-mathematical situations in which notions such as justification, evidence, conclusive evidence make sense; notions which have been central to epistemology since its inception, and especially so after Gettier \cite{Gettier1963}. 
Indeed, the notions of proof and conclusive evidence are perfectly normal in various non-mathematical domains, for instance in the context of legal standards establishing guilt or tort.

A number of BHK-style formalisms, e.g.\ Justification Logic (cf.\ \cite{Artemov2001,Artemov2008,Artemov2012b}), have been developed which study the usual logical propositions along with justification assertions  
\[ \mbox{\it t is a justification for F} .\]
Indeed, Justification Logic first appeared as the Logic of Proofs which studied formal mathematical proofs, and was able to fairly represent the mathematical BHK semantics \cite{Artemov2001}. 
However, very soon it became clear that the same logical apparatus represents the principles of justification and evidence at large, which led to general purpose logics of justification with numerous interpretations far beyond its area of origin, mathematical proofs. It is clear, then, that our intuitionistic epistemic framework, and \iel\ in particular, \emph{can} model non-mathematical epistemic situations as well. Moreover, with the built-in notion of verification, \iel-like systems offer a more expressive logical tool for studying evidence-based knowledge and beliefs in a general setting than classical epistemic logic. The embedding results in section~\ref{sec:embed} open the door to formally connecting \iel\ with justification logics and their numerous interpretations, both mathematical and non-mathematical. 

\section{Conclusion}\label{sec:conc}

Our primary goal has been to outline an intuitionistic view of belief and knowledge by articulating their basic principles within the context of the BHK semantics, and to provide a formal foundation for further studies in intuitionistic epistemology. 

We have argued that the co-reflection principle
    \[
    A \imp \bk A
    \]
is foundational for a properly intuitionistic epistemology, valid for any intuitionistic epistemic state.
Likewise, and for virtually the same reason, the reflection principle $\bk A \imp A$ is so strong intuitionistically that it is invalid.

We have sought to show that the counter-intuitiveness of both claims is merely apparent.
In particular the loss of reflection would seem to rule out our considerations as being about knowledge at all.
We argue that, on the contrary, instead of losing the ability to reason about knowledge that we have in fact gained a more discriminating perspective, one which, via the embedding of classical epistemic logic into intuitionistic epistemic logic, can accommodate all classical epistemic reasoning as well as make distinctions not possible classically.
The prime example being the distinction between reflection and intuitionistic reflection 
\[ \bk A \imp \neg \neg A .
\]

On the basis of the BHK semantics we have outlined the basic intuitionistic systems of belief and knowledge, \ielm\ and \iel.
Of course, we do not mean to rule out extensions of the systems we outline.  
Nothing prohibits additions and refinements; this is a beginning and we hope that this paper will stimulate further research in this area.
 
\section{Acknowledgments}

The authors are grateful to Sam Buss, Lev Beklemishev, Thierry Coquand, Dirk van Dalen, Walter Dean, Melvin Fitting, Martin Hyland, Vladimir Krupski, Hidenori Kurokawa, Anil Nerode, Elena Nogina, Alessandra Palmigiano Graham Priest, and Junhua Yu for inspiring discussions and useful suggestions. 
We would also like to thank audiences at conferences and seminars in New York, Boulder, Rome, Los Angeles, Madrid, Vienna, Bucharest, Mexico City, Oberwolfach, Delft, Helsinki where versions of this paper were presented and discussed.

\begin{appendices}

\section{Completeness of Intuitionistic Epistemic Logic}\label{sec:comp}

We show that \ielm\ and \iel\ are complete with respect to the classes of corresponding models. 
The proof is a straightforward extension of the standard completeness proof for {\sf IPC}.  
We will first present a proof for $\ielm$ and then show how to modify it for $\iel$.

First we define the notion of a \emph{prime theory} over ${\mathcal L}$.

\begin{definition}\label{df:theory}
 A set of formulas, $\gg$, is a \emph{theory} if it is closed under $\proves$ in $\mathcal L$. That is, for any $A$, if $\gg \proves A$ then $A \in \gg$. 
 A set of formulas, $\gg$, is \emph{prime} if $A \lor B \in \gg$ implies that either $A \in \gg$ or $B \in \gg$. 
\end{definition}

The following lemma is established by the standard Henkin construction
\begin{lemma}\label{lm:Lindenbaum}
For a set of formulas $\gg$ and formula $A$, if $\gg \nvdash A$ then there exists a prime theory $\gd$, such that $\gg \subseteq \gd$ and $A\notin \gd$.  
\end{lemma}

\begin{theorem}[Completeness of $\mathcal{L}=\ielm, \iel$]\label{th:ielcomp} If a formula $A$ holds in each mo\-del of ${\mathcal L}$, then $\mathcal{L}\proves A$.
\end{theorem}

\begin{proof}

We now define the canonical model. 
\begin{definition}
 The canonical model is a quadruple $\langle W, R, \forces, E\rangle$ such that: 
\begin{itemize}
 \item $W$ is the set of all consistent prime theories; 
 \item $\gg R \gd$ iff $\gg \subseteq \gd$; 
 \item $\gg E \gd$ iff $\gg_\bk \subseteq \gd$ where $\gg_\bk = \{ A\mid \bk A \in \gg \}$;
 \item $\gg \forces p \text{ iff } p \in \gg$, for a propositional letter $p$. 
\end{itemize}
\end{definition}
\begin{lemma}\label{df:cmodel}
The canonical model is a model for $\mathcal{L}$.
\end{lemma}
\begin{proof}
 Clearly, $\subseteq$ is a partial order, hence so is $R$. We need to show that $E$ is a binary relation meeting the following conditions: 

\begin{enumerate}[noitemsep]
 \item $E \subseteq R$,
 \item $\gg R \gd\ \  \Imp \ E(\gg )\supseteq E(\gd)$.
\end{enumerate}

 1. Assume $\gg E \gd$ and $X \ig$. Since $\gg$ contains $\mathcal{L}$,  $X \imp \bk X \ig$, hence $\bk X \ig$, but then $X \in \gd$. Since $X$ is arbitrary, $\gg \subseteq \gd$. 

 2. Assume $\gg R \gd E \Theta$. We have to show that $\gg E \Theta$, i.e.\ $\gg_\bk\subseteq\Theta$. Take $X \in \gg_\bk$, i.e. $\bk X \ig$. Since $\gg R \gd$ \  $\bk X \in \gd$, hence $X \in \gd_\bk$. Since $\gd E \Theta$ holds, $\gd_\bk \subseteq \Theta$, hence $X \in \Theta$. 

In addition, for $\mathcal{L}=\iel$ we have to show that $E(\gg)\neq\emptyset$ for each $\gg\in W$. Indeed, take such a $\gg$. We have to check that there is $\gd\in W$ such that $\gg_\bk\subseteq\gd$ and for this it suffices to secure the consistency of $\gg_\bk$ since then, by Lemma~\ref{lm:Lindenbaum}, a desired $\gd$ exists. Suppose $\gg_\bk$ is not consistent. Then for some $A_1,A_2,\ldots,A_n\in\gg_\bk$ 
\[ \proves (A_1\wedge A_2\wedge\ldots\wedge A_n)\imp\bot . \]
By $\bk$-necessitation and some modal reasoning, 
\[ \proves (\bk A_1\wedge \bk A_2\wedge\ldots\wedge \bk A_n)\imp\bk \bot . \]
Since $\iel\proves\bk\bot\imp\bot$, 
\[ \proves (\bk A_1\wedge \bk A_2\wedge\ldots\wedge \bk A_n)\imp \bot . \]
Since $\bk A_1,\bk A_2,\ldots,\bk A_n\in\gg$, $\gg$ is inconsistent - a contradiction.

\end{proof}

\begin{lemma}[Truth Lemma]\label{lm:tlemma}
For any formula $X$, $\ \gg \forces X \ \BImp X \in \gg$.
\end{lemma}
\begin{proof}  By induction on the construction of $X$. The propositional cases are standard, we check the epistemic case only, i.e.\ when $X$ is $\bk Y$. 

$\Imp$: Assume $\bk Y \ig$, and $\gg E \gd$, hence $Y \in \gd$. By the induction hypothesis $\gd \forces Y$. Since $\gd$ is arbitrary this holds for any state $E$-accessible from $\gg$ hence $\gg \forces \bk Y$. 

$\cImp$: Suppose $\bk Y \notin \gg$, in which case $\gg_\bk \nproves Y$. Suppose otherwise (i.e.\ suppose $\gg_\bk  \proves Y$), then $A_1 \dots A_n \proves Y$ for some $A_i \ig_\bk$. By the deduction theorem $\proves A_1 \land \dots \land A_n \imp Y$. Hence $\proves (\bk A_1 \land \dots \land \bk A_n) \imp \bk Y$. Now $\bk A_1 \dots \bk A_n \in \gg$, hence $\gg \proves \bk Y$. Since $\gg$ is a theory, $\bk Y \ig$, which is a contradiction. Hence $\gg_\bk  \nproves Y$. By Lemma \ref{lm:Lindenbaum} there is a prime $\gd$ such that $\gg_\bk  \subseteq \gd$ and $Y \notin \gd$. By the induction hypothesis $\gd \nforces Y$ hence $\gg \nforces \bk Y$. 

\end{proof}

To finish the proof of Theorem~\ref{th:ielcomp}, assume ${\mathcal L}\nproves X$, which can be read as $\emptyset\nproves X$.
  By Lemma \ref{lm:Lindenbaum} there is a prime $\gd$ s.t.\ $X \notin \gd$; such a $\gd$ is consistent. 
  By the Truth Lemma, in the canonical model $\gd \nforces X$, so ${\mathcal L}\nforces X$. 
  
\end{proof}

\section{The Truth Condition on Knowledge}\label{sec:itcond}

We stated above (section \ref{sec:knf}) that $\neg \bk \bot$ is the simplest candidate for expressing the truth condition on knowledge in an intuitionistic manner. 
In the presence of co-reflection each of the alternatives to intuitionistic reflection are equivalent (section \ref{sec:IEL}). 
It is easy to show that in the absence of co-reflection we get the following hierarchy, from strongest to weakest.

\bigskip\par

    \begin{figure}[H]
      \centering
    $\bk A \imp A$

    \bigskip

    $\Downarrow$

    \bigskip

    \mbox{$\neg(\bk A \land \neg A) \ \ \Leftrightarrow \ \ (\bk A \imp \neg\neg A) \ \ \Leftrightarrow \ \ \neg\neg(\bk A \imp A) \ \ \Leftrightarrow \ \ (\neg A \imp \neg \bk A)$}

    \bigskip

    $\Downarrow$

    \bigskip

    $\neg \bk \bot$
    \caption{Heirarchy of Intuitionistic Truth Conditions}

    \end{figure}

These dependencies can be checked in the logic $\Intk$, the intuitionistic analogue of the classical modal logic \sft{K}. 
This is \ielm\ without co-reflection but with the necessitation rule. For models of $\Intk$, see \cite{Bozic1984, Gabbay2003}.

\end{appendices}

\printbibliography

\end{document}